\documentclass{article}

\pdfoutput=1
\usepackage{authblk} 
\usepackage{wrapfig}
\usepackage{graphicx}
\usepackage{amsmath}
\usepackage{bbm}
\usepackage{amsfonts}
\usepackage[margin=2.5cm]{geometry}
\usepackage{multicol}
\usepackage{amsthm}
\usepackage{graphicx}
\usepackage{multicol}
\usepackage{hyperref}
\usepackage[utf8]{inputenc} 
\usepackage[T1]{fontenc}    
\usepackage{hyperref}       
\usepackage{amsfonts}       
\usepackage{color}
\usepackage{multirow}

\newcommand{\parr}[1]{\left (#1\right )}
\newcommand{\brac}[1]{\left [#1\right ]}
\newcommand{\ip}[1]{\left \langle #1 \right \rangle }
\newcommand{\Real}{\mathbb R}

\newcommand{\bbone}{\mathbbm 1}

\renewcommand{\P}{\mathcal{P}}

\newcommand{\argmin}{\textrm{argmin}}

\def \argmin{\mathrm{argmin}}

\def \RS{\mathrm{RS}} 
\def \DS{\mathrm{DS}} 
\def \CS{\mathrm{CS}} 
\def \JAP{\mathrm{JAP}} 
\def \OLP{\mathrm{OLP}} 

\def \P{\mathcal{P}}

\newtheorem{theorem}{Theorem}
\newtheorem{lemma}{Lemma}

\newcommand{\eg}{{\it e.g.}}
\newcommand{\ie}{{\it i.e.}}

\title{Sinkhorn Algorithm for Lifted Assignment Problems}

\author{Yam Kushinsky}
\author{Haggai Maron}
\author{Nadav Dym}
\author{Yaron Lipman}
\affil{Weizmann Institute of  Science}
\date{}
\usepackage[normalem]{ulem}

\begin{document}


\maketitle

\begin{abstract}
Recently, Sinkhorn's algorithm was applied for approximately solving linear programs emerging from optimal transport very efficiently \cite{cuturi2013sinkhorn}. This was accomplished by formulating a regularized version of the linear program as Bregman projection problem onto the polytope of doubly-stochastic matrices, and then computing the projection using the efficient Sinkhorn algorithm, which  is based on alternating closed-form Bregman projections on the larger polytopes of row-stochastic and column-stochastic matrices.

In this paper we suggest a generalization of this algorithm for solving a well-known lifted linear program relaxations of the Quadratic Assignment Problem (QAP), which is known as the Johnson Adams (JA) Relaxation. First, an efficient algorithm for Bregman projection onto the JA  polytope by alternating closed-form Bregman projections onto one-sided local polytopes is devised. The one-sided polytopes can be seen as a high-dimensional, generalized version of the row/column-stochastic polytopes. 
Second, a new method for solving the original linear programs using the Bregman projections onto the JA polytope is developed and shown to be more accurate and numerically stable than the standard approach of driving the regularizer to zero. 
The resulting algorithm is considerably more scalable than standard linear solvers and is able to solve significantly larger linear programs. \vspace{-6pt}
\end{abstract}

\section{Introduction} \vspace{-6pt}

The popular Sinkhorn algorithm \cite{kosowsky1994invisible,cuturi2013sinkhorn} for optimal transport problems solves optimal transport problems extremely efficiently, at the price of a minor modification of the energy to be minimized which takes the form of an entropic regularization term. The regularized optimal transport problem can be phrased as the problem of computing the Bregman projection of a matrix onto the optimal transport polytope.  The Sinkhorn algorithm represents the optimal transport polytope as an intersection of two polytopes for which the Bregman projection has a simple closed-form solution, and then iteratively computes these projections in an alternating fashion. This results in a provably convergent algorithm for regularized optimal transport problems that is significantly more scalable than generic linear programming (LP) solvers.

In this paper we propose a Sinkhorn-type algorithms for the famous Johnson-Adams linear relaxation of the Quadratic Assignment Problem (QAP) . The QAP as introduced in Lawler \cite{lawler1963quadratic} is the problem of finding a bijection between the $n$ vertices of two graphs minimizing a quadratic energy.  Two well-known subproblems of the QAP are the traveling salesman problem and the Koopmans-Beckmann quadratic assignment problem. Approximately solving either one of these subproblems is known to be NP-hard in general \cite{QAPnpHardToApprox}. The popular Johnson-Adams (JA)  relaxation \cite{adams1994improved} for the QAP is an LP relaxation defined in a \emph{lifted} high dimensional variable space  with $O(n^4)$ variables and constraints. As a result they are often too big to solve with generic (\eg, interior point) LP solvers. We represent the \emph{Johnson Adams polytope} (JAP) as an intersection of four polytopes which we call one-sided local polytopes. We show that  computing Bregman projections onto a one-sided local polytope has an easily computable \emph{closed-form solution}. The time complexity of computing this closed-form solution is \emph{linear} in the size of the data. Based on this observation and the fact that the JAP is the intersection of four one-sided local polytopes, we propose an efficient, provably convergent Sinkhorn-type algorithm for computing Bregman projections onto the JAP, by iteratively solving one-sided problems. 

Once we have an efficient algorithm for Bregman projection onto the JAP, we can use this algorithm to optimize linear energies over these polytopes. At this point we abandon the standard regularization approach used by the Sinkhorn algorithm, and suggest an alternative process for iteratively using Bregman projections for solving the original LP. The resulting algorithm for solving the original LP is more accurate and numerically robust than the standard entropy regularization approach. 

We provide numerical experiments validating our algorithm on the standard QAP benchmark \cite{burkard1997qaplib} achieving slightly inferior results to the best known lower-bounds for these problems. We note that these best lower-bounds were achieved with a plethora of different techniques including combinatorial algorithms with exponential worst-case time complexity. We further apply our algorithm to three "real-life" anatomical datasets of bones \cite{boyer2011algorithms} demonstrating state of the art classification results, improving upon previous works and providing better classification than human experts in all but one (almost comparable) instance.

\section{Related work}

\paragraph*{Quadratic assignment problems}
Convex relaxations are a common strategy for dealing with the hardness of the QAP. Small-medium instances of the QAP ($n<30$) can be solved using branch and bound algorithms which use convex relaxations to obtain lower bounds \cite{QAPsurvey}. For larger problems the non-integer solution obtained from the relaxation is rounded to  obtain a feasible (generally suboptimal) solution for the QAP. Examples include  spectral relaxations \cite{rendl1992applications,leordeanu2005spectral} and quadratic programming relaxations over the set of doubly stochastic matrices \cite{anstreicher2001new,zaslavskiy2009path,fogel2013convex}. Lifting methods, in which auxiliary variables that represent the quadratic terms are introduced, provide linear programming (LP) relaxations \cite{adams1994improved} or semi-definite programming relaxations \cite{zhao1998semidefinite,Itay} which are often more accurate than the former methods. For example for certain classes of the QAP the worst case error of the  LP relaxations can be bounded by a multiplicative constant of $\approx 3.16$ \cite{nagarajan2009maximum}. The disadvantage of lifting methods is that they solve convex problems with $n^4$ variables in contrast with the cheaper spectral and quadratic programming methods that solve problems with $n^2$ variables. As a result, lifting methods cannot be solved using generic convex programming solvers for $n>20$. It is also possible to construct relaxations with $n^{2k}$, $k>2$ variables to achieve even tighter relaxations \cite{laurent2003comparison,adams2007level,hahn2012level} at an increased computational price.

The authors of \cite{adams1994improved,karisch1999dual} suggest to deal with the computational complexity of the large JA linear program by using a greedy coordinate ascent algorithm to solve the dual LP. This algorithm is not guaranteed to converge to the global minimum of the JA relaxations. The authors of \cite{rendl2007bounds} propose a specialized solver for a lifted SDP relaxation of QAP, and the authors of \cite{burer2006solving} propose a converging  algorithm for the JA and SDP relaxations. However both algorithms can only handle quadratic assignment instances with up to 30 points. More on the QAP can be found in surveys such as \cite{QAPsurvey}.

\paragraph*{Entropic regularization} The successfulness of entropic regularization for optimal transport linear programs has motivated research aimed at extending this method to other optimization problems. In \cite{rangarajan1996novel,rangarajan1997convergence,Justin} it is shown that regularized quadratic energies over positive matrices with fixed marginal constraints can be solved efficiently by solving a sequence of regularized optimal transport problems. Cuturi et al. \cite{cuturi2014fast} compute Wasserstein barycenters using entropic regularization. Benamou et al. \cite{benamou2015iterative} also consider Wasserstein barycenters as well as  several other problems for which entropic regularization can be applied. One of these problems is the multi-marginal optimal transport which is related to the JA linear program, although the latter is more complex as the marginals in the JA linear program are themselves variables constrained by certain marginal constraints.

\section{Approach}
\subsection{Problem statement}
The \emph{quadratic assignment problem} (QAP) is the problem of minimizing a quadratic energy over the set $\Pi=\Pi(n)$ of permutation matrices of dimension $n$:
\begin{equation}\label{e:qap}
\min_{x\in \Pi}  \quad \sum_{ij} \theta_{ij} x_{ij} + \sum_{ijkl}\tau_{ijkl}x_{ij}x_{kl}.
\end{equation}

One common and powerful approximation to the solution of \eqref{e:qap} is achieved via an LP relaxation in a lifted space. That is, \eqref{e:qap} is relaxed by replacing quadratic terms $x_{ij}x_{kl}$ with new auxiliary variables $y_{ijkl}$ to obtain
\begin{equation}\label{e:perm_lp}
\min_{(x,y)\in C}  \quad \sum_{ij} \theta_{ij} x_{ij} + \sum_{ijkl}\tau_{ijkl}\,y_{ijkl},
\end{equation}
where $C$ is the \emph{Johnson Adams polytope} (JAP) which is a convex relaxation of $\Pi$ in the lifted $(x,y)$ space:
\begin{subequations}\label{e:JA}
	\begin{align}\label{e:JA_1}
	\sum_j x_{ij} &= 1, \qquad \ \ \ \forall i \\
	\sum_i x_{ij} &= 1, \qquad \ \ \ \forall j \\ \label{e:JA_3}
	\sum_l y_{ijkl} &= x_{ij}, \qquad \forall i,j,k \\ 
	\sum_k y_{ijkl} &= x_{ij}, \qquad \forall i,j,l \\ \label{e:JA_5}
	\sum_j y_{ijkl} &= x_{kl}, \qquad \forall i,k,l \\
	\sum_i y_{ijkl} &= x_{kl}, \qquad \forall j,k,l \\ \label{e:JA_geq_0}
	x,y &\geq 0.
	\end{align}
\end{subequations}

Here $x\in\Real^{n\times n}$, $y\in\Real^{n^2\times n^2}$. It is indeed a relaxation of $\Pi$ since every  permutation $x$ satisfies $(x,y)\in \JAP$ for $y_{ijkl}=x_{ij}x_{kl}$.
For notational convenience we let $d=n^2+n^4$ and denote $(x,y)\in\Real^d$.\\

\subsection{Sinkhorn's algorithm}
Our goal is to construct efficient algorithms for solving the JA relaxation. Our method is motivated by the successfulness of the highly scalable Sinkhorn algorithm \cite{kosowsky1994invisible,cuturi2013sinkhorn} in (approximately) solving optimal transport problems. We begin by reviewing the key ingredients of the Sinkhorn algorithm and then explain how we generalize it to higher order LP relaxations, and the modifications we suggest for improving convergence.

To solve optimal transport\ (OT) problems efficiently, it is suggested in 
\cite{bregman1967relaxation,cuturi2014fast,benamou2015iterative} to add an entropic regularizer to the OT problem:
\begin{equation}\label{e:OT_entropy}
\min_{x\in \DS}  \quad \ip{\theta,x} + \beta^{-1} \sum_{ij} x_{ij} \parr{\log x_{ij} - 1},
\end{equation}
where $\beta$ is some large positive number, and  $\DS=\DS(\mu,\nu)\subset \Real^{n\times n}_{\scriptscriptstyle{ \geq 0}}$ is the set of non-negative $n\times n$ matrices with specified positive marginals $\mu,\nu \in \Real_{\scriptscriptstyle{ > 0}}^n $: 
\begin{subequations}\label{e:DS}
	\begin{align} \label{e:DS_RS}
	& \quad \sum_j x_{ij}=\mu_i, \quad \forall i\\ \label{e:DS_CS}
	& \quad \sum_i x_{ij}=\nu_j, \quad \forall j\\ \label{e:DS_geq}
	& \quad x_{ij}\geq 0, \quad \forall i,j
	\end{align}
\end{subequations}

Adding the entropy to the energy has several benefits: First, it allows writing the energy as a Kullback-Leibler divergence w.r.t.~some $z\in\Real^{n\times n}_{\scriptscriptstyle{ > 0}}$, 
\begin{equation}\label{e:kl}
\min_{x\in \DS} \quad KL(x \vert z),
\end{equation}
where $KL(x\vert z) = \sum_{ij} x_{ij} \parr{\log\frac{x_{ij}}{z_{ij}} -1}$ is the KL divergence. This turns \eqref{e:OT_entropy} into an equivalent  KL-projection problem. Secondly, it makes the energy strictly convex. Thirdly,  since the entropy's derivative explodes at the boundary of $\DS$~it serves as a barrier function which ensures that the inequality constraints \eqref{e:DS_geq} are never active, resulting in significant simplification of the KKT equations for \eqref{e:DS}; Finally, due to this simplification,  the KL-projection over the row-stochastic matrices $\RS(\mu)$ defined by \eqref{e:DS_RS}, and column-stochastic matrices $\CS(\nu)$ defined by \eqref{e:DS_CS} has a closed form solution:

\begin{theorem}\label{thm:kl_projection_rs}
	Given $z\in \Real_{\scriptscriptstyle{ > 0}}^{n\times n}$, the minimizer of
	\begin{equation}\label{e:kl_rs}
	\min_{x\in \RS(\mu)} \quad KL(x \vert z),
	\end{equation}
	is realized by the equation
	\begin{equation}\label{e:KL_projection_RS}
	x_{ij}^*=\frac{z_{ij}}{\sum_s z_{is}} \mu_i,
	\end{equation}
	that is the row normalized version of $z$. Simliarly the projection onto $\CS$ is the column normalized version of $z$.
\end{theorem}

The theorem  is proved by directly solving the KKT equations of \eqref{e:kl_rs} (see \eg~\cite{benamou2015iterative}). These observations are used to construct an efficient algorithm to approximate the solution of the regularized OT problem \eqref{e:OT_entropy} by repeatedly solving KL-projections on $\RS(\mu)$ and $\CS(\nu)$. As proved in \cite{bregman1967relaxation} this converges to the minimizer of \eqref{e:OT_entropy}. 

Following \cite{benamou2015iterative}, we note that the Sinkhorn algorithm is an instance of the Bregman iterative projection method that allows solving KL-projection problems over intersection of affine sets $C_1,C_2,\ldots,C_N$,
\begin{subequations}\label{e:KL}
	\begin{align} \label{e:KL_e}
	\min_{x \geq 0} &\quad  KL(x \vert z)\\
	\mathrm{s.t.} & \quad x\in C_1\cap C_2\cap\dots\cap C_N
	\end{align}
\end{subequations}
via alternate KL-projections on the sets $C_k$, that is
\begin{subequations}\label{e:bergman}
	\begin{align}
	x_0 & = z\\
	x_n & = \argmin_{x\in C_{\mathrm{mod}(n-1,N)+1}}KL(x\vert x_{n-1}), \quad n\geq 1
	\end{align}
\end{subequations}
In \cite{bregman1967relaxation} it is shown that this procedure is guaranteed to converge, under the conditions that: (i) the feasible set of \eqref{e:KL}, $C=\cap_i C_i$ contains a vector whose entries are all strictly positive , \ie, it is \emph{strictly feasible}; and (ii) All entries of the minimizer of \eqref{e:KL_e} over each $C_i$ are strictly positive . In fact, in the case of KL-divergence (in contrast to the general Bregman divergence dealt with in \cite{bregman1967relaxation}), condition (ii) can be proved from (i) using the fact that the derivatives of the KL-divergence blow-up at the boundary of the set defined by $x\geq 0$. 
Lastly, condition (i) is satisfied in all the problems we discuss in this paper. For example, $\DS$ contains a feasible interior point $x=\frac{1}{n}\mu \nu^T >0$. 

\subsection{Approach}
Our approach for solving lifted assignment problems is based on two main components: 
The first component is an efficient computation of KL projections onto the lifted JA polytope using alternating projections. While Bregman iterations can always be used to solve problems of the form \eqref{e:KL}, the performance of the method greatly depends on the chosen splitting of the feasible convex set $C$ into convex subsets $C_i$, $i=1,\ldots, N$. Generally speaking a good splitting will split $C$ into a small number of sets, where the KL-projection on every set is easy to compute. The successfulness of the optimal transport solution can be attributed to the fact that the feasible set $C=\DS(\mu,\nu)$ is split into only two sets $C_1=\RS(\mu)$, $C_2=\CS(\nu)$, and the projection onto each one of these sets has a closed form solution. We will use Bregman iterations  to approximate the solution of the JA relaxation of the QAP. We split the feasible sets of these relaxations into four sets, so that the projection on each one of these sets has a closed-form solution. For comparsion, note that the standard alternating type method for the Johnson Adams relaxation needs to solve multiple  linear programs in $n^2$ variables in each iteration \cite{adams1994improved,karisch1999dual,rostami2014revised} instead of computing closed form solution for each iteration, and is not guaranteed to converge. Our algorithm for computing KL-projections onto lifted polytopes is described in Section~\ref{sec:KLontoLifted}.

The second component of our approach is using the KL projections onto the lifted JA polytope  for approximating the solution of the linear program \eqref{e:perm_lp}. The approximation provided by the standard Sinkhorn algorithm described above is known to be suboptimal since in practice the parameter $\beta$ in \eqref{e:OT_entropy} cannot be chosen to be very large due to numerical instabilities. We propose an alternative method for approximating the solutions of the linear program by iteratively solving a number of KL-projection problems. We find that this method gives a good approximation of the solution of the linear program in a small number of iterations. This method is discussed in Section~\ref{sec:lin2KL}.


\section{KL-Projections onto lifted polytopes}\label{sec:KLontoLifted}
We consider the problem of minimizing 
$$ KL(x\vert z) +KL( y \vert w),$$ 
where $x,z\in\Real^{n\times n}$ and $y,w\in\Real^{n^2\times n^2}$, $w,z>0$, over the  JAP using alternating Bregman iterations. The main building block in this algorithm is defining the
  \emph{one-sided local polytope} ($\OLP$):
\begin{subequations}\label{e:olp}
	\begin{align}
	\sum_j x_{ij} &= 1, \qquad \ \ \ \forall i \\
	\sum_l y_{ijkl} &= x_{ij}, \qquad \forall i,j,k \label{e:olp_Yconstraints}\\
	x,y & \geq 0          
	\label{e:local_poly_positivity}
	\end{align}
\end{subequations}
and observing that the $\JAP$ is an intersection of four sets, which are, up to permutation of coordinates, $\OLP$ sets: 
Denote $y_{ijkl}^\diamond=y_{klij}$ and define $\OLP^\diamond$ as the set of $(x,y)$ satisfying $(x,y^\diamond)\in\OLP$. Next denote $y^T_{ijkl}=y_{jilk} $ and define $\OLP^T $ to be the set of $(x,y)$ satisfying $(x^T,y^T) \in \OLP $. Denote by $\OLP^{T\diamond}$ the set of $(x,y)$  satisfying $(x^T,(y^T)^\diamond) \in \OLP $. Then we obtain
\begin{equation}\label{e:TLPnTLPTn_decomp}
\JAP=\OLP\cap\OLP^\diamond\cap \OLP^T\cap \OLP^{T\diamond}.
\end{equation}
We show that there is a closed form formula for the KL-projection onto the OLP polytope. 
The derivation of this formula will be presented in the next subsection. Thus by applying
 Bregman iterations iteratively to the four OLP sets as in \eqref{e:bergman}, we  are guaranteed to converge to the KL-projection onto the lifted polytope, providing that the JAP is strictly feasible. This is indeed the case; an example of a strictly feasible solution is $x=\frac{1}{n} \bbone \bbone^T, y=\frac{1}{n^2} \bbone \bbone^T$  where $\bbone$ denotes the vector of all ones in the relevant dimension.
\subsection{KL-Projections onto the one-sided local polytope}

We now compute the closed-form solution for KL-projections over the one-sided local polytope ($\OLP$) defined in \eqref{e:olp}. Namely for given $(z,w)\in\Real^d_{\scriptscriptstyle{ > 0}}$ we seek to solve 
\begin{equation}\label{e:kl_olp}
\min_{(x,y)\in \text{OLP}} KL(x \,\vert\, z) +  KL(y  \, \vert \,w),
\end{equation}

\begin{theorem}\label{thm:KL_projection_LP1}
Given $(z,w)\in \Real^d_{\scriptscriptstyle{ > 0}}$, the minimizer of \eqref{e:kl_olp} is given by the equations:
\begin{subequations}\label{e:kl_olp_solution}
\begin{align}
q_{ij}&=\exp \parr{\frac{\sum_k \log (\sum_s w_{ijks}) +  \log z_{ij}}{n+1}} \\
x_{ij} & = \frac{q_{ij}}{\sum_j q_{ij}} \\
y_{ijkl} &= x_{ij}\frac{w_{ijkl}}{\sum_s w_{ijks}}
\end{align}
\end{subequations}
\end{theorem}
\begin{proof}
The proof is based on two applications of Theorem \ref{thm:kl_projection_rs}.
First, we will find the optimal $y$ for any fixed $x$. Indeed, fixing $x$ decomposes \eqref{e:kl_olp} into $n\times n$ independent problems, one for each pair of indices $i,j $ in \eqref{e:olp_Yconstraints}. Each independent problem can be solved using the observation that the matrix  $u_{kl}=y_{ijkl}$ is in $\RS(\mu)$ where $\mu$ is the constant vector $\mu=x_{ij}\bbone$, where $\bbone$ denotes the vector of ones.  Thus using Theorem~\ref{thm:kl_projection_rs} 
$$y_{ijkl}=u_{kl} = x_{ij}\frac{w_{ijkl}}{\sum_s w_{ijk s}}.$$
Now we can plug this back in \eqref{e:kl_olp} and end up with a problem in the variable $x$ alone. Indeed, 
\begin{align*}
 KL(x\,\vert\, z)+KL(y\,\vert\, w) &=\sum_{ij} \brac{  KL(x_{ij} \, \vert \, z_{ij} )+ \sum_{kl} \frac{w_{ijkl}}{\sum_s w_{ijk s}} KL\Big(x_{ij} \, \Big\vert \, \sum_s w_{ijk s} \Big)} \\
&=(n+1)\sum_{ij} KL \Bigg( x_{ij} \, \Bigg \vert\, \exp\parr{\frac{
                \log z_{ij} + \sum_k \log (\sum_s w_{ijks})}{n+1}}   \Bigg),
\end{align*}
where in the second equality we used the following (readily verified) property of KL-divergence $a_k,b_k>0$: 
$$\sum_{k} a_{k} KL(x\,\vert\, b_{k}) = \big(\sum_{k}a_{k}\big) KL \parr{x \, \Big\vert \exp \parr{\frac{\sum_k a_k \log b_k}{\sum_k a_k}} }.$$
Finally, we are left with a single problem of the form \eqref{e:kl_rs} and applying Theorem \ref{thm:kl_projection_rs} again proves \eqref{e:kl_olp_solution}.
\end{proof}

\paragraph{Incorporating zeros constrains} The JA relaxation stated above can be strengthened by noting that for  permutations $x\in \Pi$ there exists exactly one non-zero entry in each row and column and therefore $x_{ij}x_{il}=0$ and $x_{ji}x_{li}=0 $ for all $j\ne l$. In the lifted LP formulation this implies $y_{ijil}=0$ and $y_{jili}=0 $ for all $i$ and $j\ne l$.  These constraints (which are sometimes called \emph{gangster constraints}) are part of the standard JA relaxation. They  can be incorporated seamlessly in our algorithm as we will now explain.

Denote multi-indices of $y$ by $\gamma$  and let $\Gamma$ be the set of multi-indices $\gamma$ for which the  constraint $y_\gamma=0 $ is to be added. We eliminate the zero valued variables from  the objective \eqref{e:perm_lp} and the constraints defining the polytope $C$,  and rewrite them as optimization problem in the variables $x$ and $(y_\gamma)_{\gamma \not \in \Gamma} $. We then consider $KL$-projections only with respect to these  variables, and use the same Bregman iteration scheme described above for the reduced variables. The only modification needed to the algorithm is a minor modification to the formula \eqref{e:kl_olp_solution}, where $w$ is replaced with $\bar w$ which satisfies $\bar w_\gamma =0$ if $\gamma \in \Gamma$ and $\bar w_\gamma=w_\gamma $ otherwise. 

 We note that also with respect to the reduced variables the strengthened relaxations are strictly feasible so that the alternating KL-projection algorithm converges. An example of a strictly feasible solution in the $\JAP$ being
$$\quad (x,y)=\frac{1}{|\Pi |} \sum_{x \in \Pi} (x,y(x)), $$    
where $y(x)$ is defined via $y_{ijkl}=x_{ij}x_{kl}$.

\section{From linear programs to KL projections}\label{sec:lin2KL}

The JA relaxation of the  QAP, and in fact all linear programs,  are  of the general form
\begin{equation} \label{e:linear}
	\min_{v \in \P} \quad c^Tv
\end{equation}
where  $\P$ is a standard polytope
$$\P=\{v\ \vert \   v \geq 0, \ Av=b\}.$$
containing a strictly feasible solution. We want to approximate a solution of the linear program using KL-projections onto $\P$.  
The most common strategy for doing this (\eg, \cite{bregman1967relaxation,cuturi2013sinkhorn,benamou2015iterative}) which we already described above, is  regularizing  \eqref{e:linear} by adding a KL-divergence term with a small coefficient $\beta^{-1}$ and solving 
\begin{equation} \label{e:regularized}
v_\beta^*= \mathrm{arg}\min_{v \in \P} c^Tv+\beta^{-1} KL(v|u_0)=\mathrm{arg}\min_{v \in \P}KL(v|u_0 \cdot \exp(-\beta c)) \end{equation}   	
Here $u_0$ is some constant positive vector, often chosen as $u_0=\bbone$, and $\cdot$ denotes elementwise multiplications. As our notation suggests, these regularized problems are strictly convex and hence have a unique minimizer $v_\beta^*$, which converges in the limit $\beta\rightarrow \infty$ to the minimizer of \eqref{e:linear} with maximal entropy \cite{benamou2015iterative}. We will call this algorithm for approximating the solution of the linear program \eqref{e:linear} the \emph{regularization method}. This approach encounters two difficulties:
\begin{enumerate}
	\item  \textbf{Underflow/overflow} occurs for large values of $\beta$.  
	\item  \textbf{Slow convergence rate} of the Bregman iterations for large values of $\beta$. This  phenomenon can be explained by the fact that Sinkhorn's algorithm can find an $\epsilon$ approximate solution in $O(n^2 \log n \epsilon^{-3}) $ time  \cite{altschuler2017near}. Thus for a fixed error rate of $\epsilon$ fast approximation by Sinkhorn's algorithm is possible, but the rate of convergence grows polynomially in $\epsilon^{-1}$ instead of logarithmically as in the case of interior point algorithms. As a result taking $\epsilon$ to be very small can lead to very long computations. 
	\end{enumerate}
 The underflow/overflow encountered for large values of $\beta$, and methods to overcome it, can be understood by considering the KKT equations of \eqref{e:regularized}: Using the fact that the unique minimum of \eqref{e:regularized} can be shown to be strictly positive, the KKT conditions amount to solving the  following equations for $v,\lambda$: 
	\begin{subequations}\label{e:KKT}
	\begin{align}
	v&=u_0 \cdot \exp(-\beta c)\cdot \exp(-A^T \lambda) \label{e:unstable}\\
	Av&=b
	\end{align}
	\end{subequations}
As $\beta$ increases the entries of $\exp(-\beta c)$ become very large numbers or very close to zero (depending on the sign of the entries), which leads to numerical overflow/underflow. 

One natural approach (which we will not use) for approximating the solution of linear programs by KL-projections which avoids underflow/overflow is the  \emph{proximal method} approach \cite{censor1992proximal,chen1993convergence}: This method proposes an iterative process beginning  with some initial guess $v_0>0$ and then solving 
\begin{equation} \label{e:PMD}
v_{k+1}=\underset{v \in \P}{\mathrm{argmin}} \, c^Tv+\beta_0^{-1} KL(v|v_k).
 \end{equation}
The advantage of this algorithm over the previous one is that it converges as $k \rightarrow \infty$ even when $\beta=\beta_0$ is held fixed and as a result the coefficients of the KKT equation
 of \eqref{e:PMD} do not explode or vanish.  On the other hand, this algorithm requires solving multiple KL-projection problems in order to converge. In our experiments in the context of the linear relaxations for the QAP  this method required many hundreds of KL-projection onto the JAP to converge (we call these iterations external iterations). As each KL projection onto the JAP polytope  typically requires several hundered (closed form) projections onto OLPs (we call these projections internal iterations), this algorithm becomes rather slow.  
 
 Accordingly we will propose a new method for approximating linear programs by KL-projections which will only require a small number of external iterations in order to converge, and still avoids underflow/overflow. The inspiration for this method comes from the following observation on the relation between the proximal method and the regularization method: 
 \begin{lemma}\label{lem:slower}
 If the proximal method and regularization method  are  initialized from the same point $(u_0=v_0 )$ then the solution obtained from the proximal method  with fixed $\beta_0$ after $k$ iterations, is equal to the solution of the regularization method when choosing $\beta=k\beta_0$ (that is $v_k=v_{k\beta_0}^* $). 
 \end{lemma}   
 \begin{proof}
 By induction on $k$. For $k=1$ the claim is obvious since in this case the equation \eqref{e:PMD} determining $v_{\beta_0}^*$  and the equation \eqref{e:regularized} determining $v_1 $ are identical. Now assume correctness for $k$, then according to \eqref{e:unstable} we have for some $\lambda_k$, 
 $$v_k=v_{k \beta_0}^*=v_0 \cdot \exp(-k\beta_0 c)\cdot \exp(-A^T \lambda_k)  $$
 By replacing $u_0 $ and $\beta$ in \eqref{e:regularized} with $v_k$ and $\beta_0$ we obtain that the KKT equations for \eqref{e:PMD} are of the form 
 \begin{subequations}
 	\begin{align}
 	v&=v_k \cdot \exp(-\beta_0 c)\cdot \exp(A^T \lambda)=u_0 \cdot \exp(-(k+1)\beta_0 c) \cdot \exp(A^T(\lambda+\lambda_k)) \label{e:two_formulations} \\
 	Av&=b
 	\end{align}
 	\end{subequations}
 	and thus the solution $v_{k+1} $ to this equation is identical to the solution of \eqref{e:KKT} with $\beta=(k+1)\beta_0 $. 
 \end{proof}
The lemma shows that the proximal method  can be interpreted as a method for improving the conditioning of the KKT equations of \eqref{e:regularized} for large values of $\beta$ by using solutions for smaller values of $\beta$ (\ie, solving $k$ iterations with $\beta_0$ is equivalent to solving one iteration with $k\beta_0\gg \beta_0$). The proof suggests other methods for exploiting solutions for small values of $\beta$ in order to solve \eqref{e:regularized} for large values of $\beta$. For example given $v_1,\ldots, v_k$, we can define $u_0$ to be
\begin{equation}
\label{e:square}
u_0=v_k \cdot v_k.
\end{equation}
Another possible choice, which is the choice we use in practice, is
\begin{equation} \label{e:last}
u_0=v_k\cdot v_{k-1}\cdot \ldots \cdot v_0 . 
\end{equation} 
We then have
\begin{lemma}\label{lem:faster}
	Let $v_k$ be defined as in \eqref{e:square} or \eqref{e:last}, If $u_0=v_0=\bbone$ then $v_k=v_{2^{k-1}\beta_0}^* $. 
\end{lemma}  
Thus, as in the previous algorithm, the proposed algorithm computes a solution $v_k$ which is in fact identical to $v_{\beta(k)}^* $ for some monotonely increasing function $\beta$. In the proposed algorithm $\beta(k)$ grows exponentially, while in the previous algorithm $\beta(k)$ grew linearly. As a result we can obtain a high-quality solution for the JA relaxation using only a small number of external iterations (around 15 in experiments we performed). 
\begin{proof}
We prove the lemma for the update rule defined in \eqref{e:last}. The proof is similar to the proof of the previous lemma.  For $k=1$ it is clear that $v_1$ and $v_{\beta_0}^* $ solve the same equation and hence are equal. Now we assume correctness of the claim for all $j \leq k $ and prove it for $k+1$. By assumption and \eqref{e:unstable}, for all $j \leq k$  there is a vector $\lambda_j$ so that  
$$v_j=\bbone \cdot \exp(-2^{j-1}\beta_0 c)\cdot \exp(-A^T \lambda_j)  $$
and therefore the KKT equations obtained by replacing $u_0 $ with $v_k$ are of the form \begin{subequations}
	\begin{align}
	v&= \exp(-\beta_0(1+\sum_{j=1}^k 2^{j-1}) c) \cdot \exp(A^T(\lambda+\sum_{j=1}^k\lambda_j)) \\
	Av&=b
	\end{align}
\end{subequations}
Thus the solution $v_{k+1}$ to this equation is identical to $v_{2^{k}\beta_0}^*$.
\end{proof}

To summarize, we state our full algorithm for computing the lifted linear relaxations of the QAP: We set $\beta=1$, and $u_0=(x_0,y_0)=(\bbone,\bbone) $. We then solve \eqref{e:regularized} using alternating Bregman projections onto the lifted polytopes (OLP) as described in Section~\ref{sec:KLontoLifted} and denote the solution by $v_1=(x_1,y_1)$. In general, we obtain $v_{k+1}$  by solving \eqref{e:regularized} with $\beta=1$ and $u_0=v_0 \cdot v_1 \cdot \ldots \cdot v_k $. In each external iteration we preform alternating Bregman projections until we recieve a solution $v_k$ which satisfies all the constraints up to a maximal error of $\epsilon$. We preform external iterations until the normalized difference between the energy of $v_{k+1}$ and $v_k$ is smaller than $\epsilon$. In our experiments we use $\epsilon=10^{-2} $.

\section{Results}
\paragraph{Comparison with interior point solvers}
\begin{wrapfigure}[12]{r}{0.3\columnwidth}
	\vspace{-1em}
	\includegraphics[width=0.3\columnwidth]{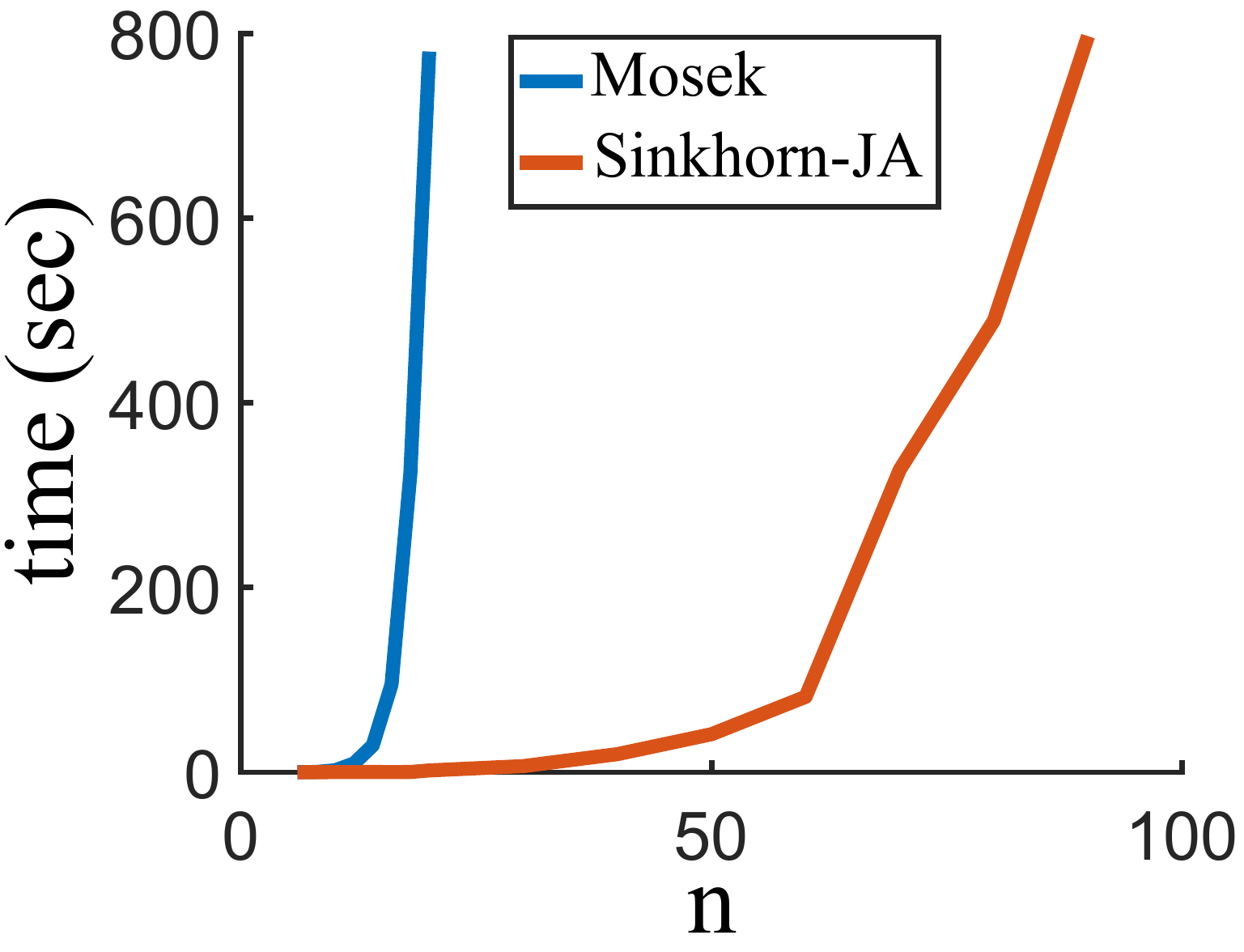}{}
	\vspace{-2em}
	\caption{\small Timing comparsion of the Sinkhorn-JA algorithm with the Mosek interior point solver.}
	\label{fig:timing}	
\end{wrapfigure}
We compare the timing of our algorithm for solving the JA algorithm (denoted by Sinkhorn-JA) with Mosek \cite{andersenmosek} which is a popular commercial interior point solver. We ran both algorithms on randomly generated quadratic assignment problems, with varying values of $n$ until they required more than ten minutes to solve the relaxations. Both algorithms were run with a single thread implementation on a Intel Xeon CPU with two 2.40 GHz processors. As can be seen in Figure~\ref{fig:timing} solving the JA relaxation with $n=20$ using Mosek takes over ten minutes. In a similar time frame we can approximately solve the JA relaxation  for problems with $n=90 $.

\paragraph{Quadratic assignment}
We evaluate our algorithm's performance for the JA relaxation using QAPLIB \cite{burkard1997qaplib}, an online library containing several data sets of quadratic assignment problems, and   provides the best known lower  and upper bounds for each problem. Many of the problems have been solved to optimality, in which case the lower bound and upper bound are equal.

In figure~\ref{fig:qaplib} we compare the upper and lower bound obtained from the proposed algorithm, with the lower and upper bound in QAPLIB. The energy of our solution for the JA relaxation provides us with a lower bound for the QAP. We obtain an upper bound by projecting the solution $x$ we obtain from the algorithm onto the set of permutations using the projection procedure of \cite{maron2018concave}  and computing its energy.  As can be seen in Figure~\ref{fig:qaplib}, for the \textbf{bur}, and \textbf{chr} datasets we achieve a very tight lower bound (3 digits of accuracy) and for the \textbf{lipa} dataset we achieve accurate solutions for the entire set. In total we achieve 19 accurate solutions (zero energy gap), and 36 lower bounds, with up to 2 digits of accuracy. For the rest of QAPLIB we achieve reasonable results. We note that the QAPLIB bounds were achieved using a rather large collection of different algorithms that are typically far slower than our own and have worst case exponential time complexity.

\begin{figure}
       \includegraphics[width=\columnwidth]{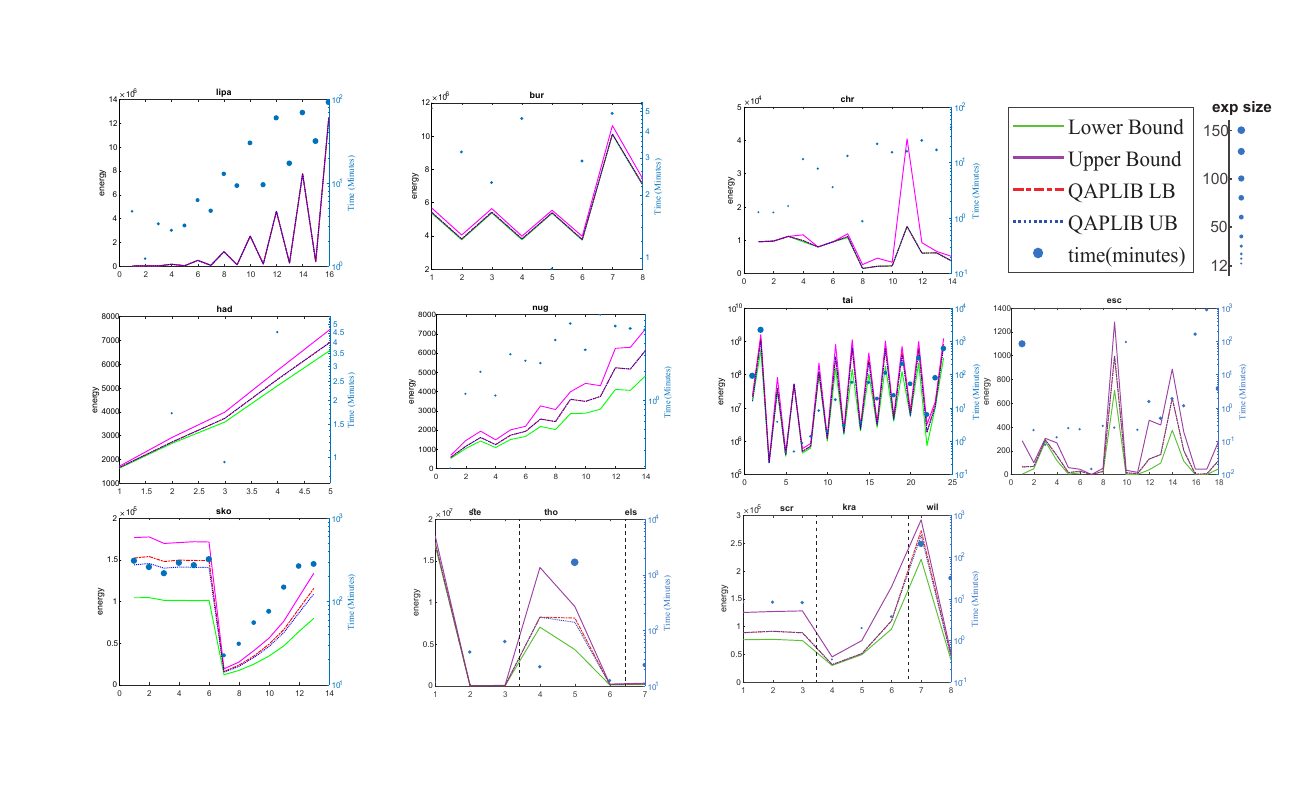}{}
       \vspace{-4em}
       \caption{\small Comparison of the upper bounds and lower bounds obtained from the Sinkhorn-JA algorithm with the best known upper bounds and lower bounds in the QAPLIB library.   }
       \label{fig:qaplib}
\end{figure}

\paragraph{Anatomical datasets} We applied our approach for the task of classification of anatomical surfaces. We considered three datasets, consisting of three different primate bone types \cite{boyer2011algorithms}:  (A) 116 molar teeth, (B) 61 metatarsal bones (C) 45 radius bones.   On each surface we sampled 400 points using farthest point sampling. We first found a correspondence map for the first 50 points using our algorithm, and then we used this result  as an initialization to \cite{maron2018concave} in order to achieve correspondences for all 400 points.  Finally we use the computed correspondences and calculate the Procrustes distance \cite{schonemann1966generalized} between each pair of shapes as a dissimilarity measure. A representative   example is shown in the inset. Note that in this case the teeth are related by an  orientation reversing map which our pipeline recovered.
\begin{wraptable}[11]{r}{0.4\columnwidth}
	\vspace{-1em}		
	\includegraphics[width=0.4\columnwidth]{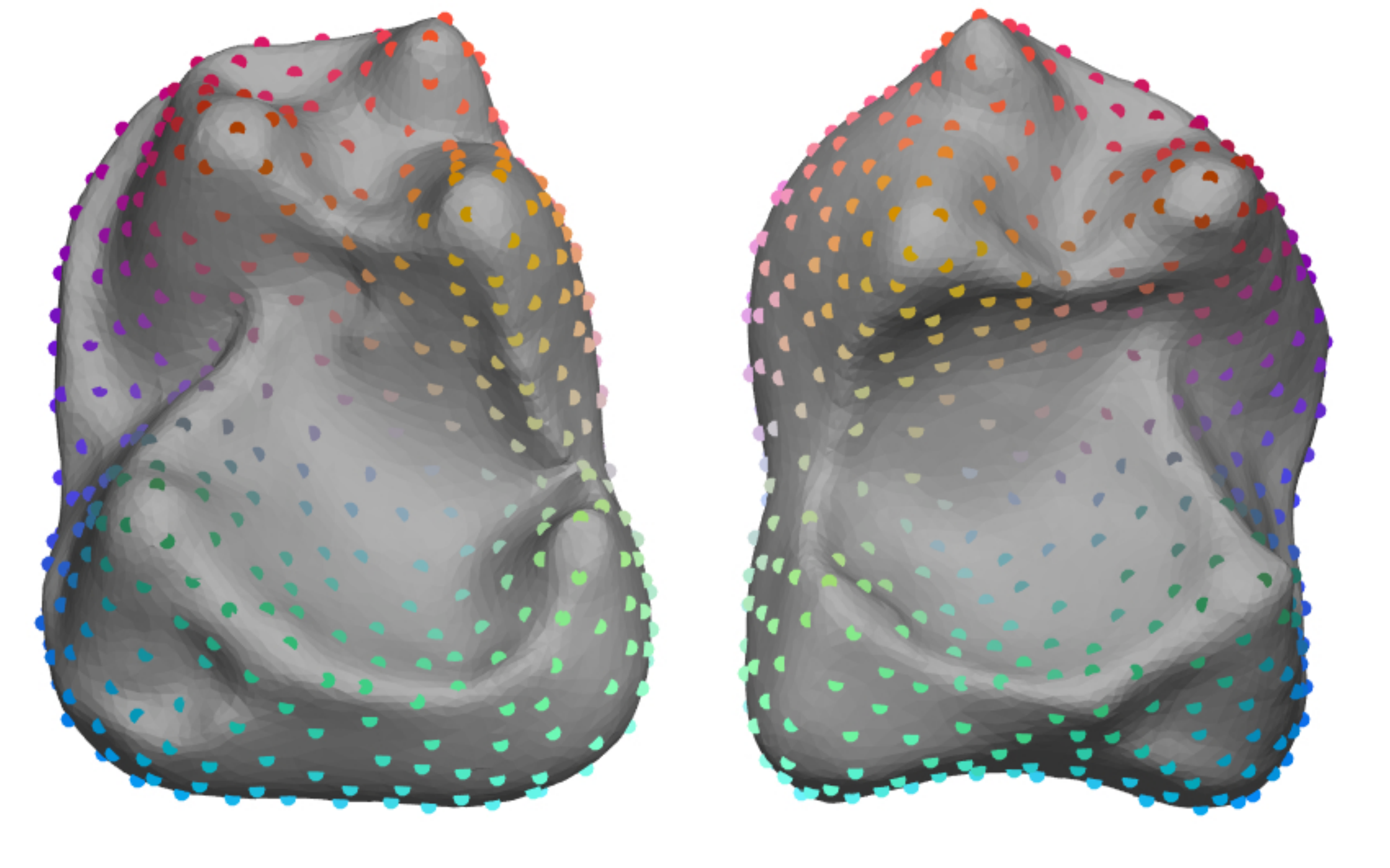}{}	
\end{wraptable}  To evaluate our algorithm, we calculate the dissimilarity measure for every two meshes in a set and use a "leave one out" strategy: each bone is assigned to the taxonomic group of its nearest neighbor among all other bones. The table below shows successful classification rates (in \%) for the three bone types and three different classification queries.
 For the initial 50 points matching we note that our algorithm is very accurate:  the normalized gap ($\frac{\text{energy} - \text{projected energy}} {\text{energy}}$) is less than 0.01 for $90\%$ of the cases. We compared our method with the  convex relaxation method of \cite{maron2016point} and the performance of human experts as reported in \cite{boyer2011algorithms}. Note that our algorithm achieves state of the art results on all but one experiment. We also compared our method with an alternative baseline method where we  match the 400 points using \cite{maron2018concave} initialized with $\frac{1}{n}\bbone\bbone^T$ and found that our algorithm achieves significantly more accurate results.



\begin{tabular}{ |p{2cm}||p{3cm}||p{1cm}|p{1cm}|p{1cm}|p{1cm}||p{1cm}||p{1cm}|  }
	\hline
	\multicolumn{8}{|c|}{data sets} \\
	\hline
	classification & \multicolumn{1}{|c|} {algorithm} & \multicolumn{2}{|c|} {Teeth}& \multicolumn{2}{|c|}{1st Metatarsal } & \multicolumn{2}{|c|}{Radius}\\
	\hline
	- &- & No. & \% & No. &\% &No. &\% \\
	\hline
    \multirow{2}{*}{genera} &Sinkhorn-JA    &  99   & \textbf{93.9}  & 59 & 83.0 & 45 &\textbf{84.44} \\
							& PM-SDP &  99  & 91.9  & 59 & 76.6 &45 & 82.44  \\
    						& Human-expert &  99  & 91.9  & 59 & \textbf{88.1} &45 & 77.8  \\
	\hline
	\multirow{2}{*}{Family}	&Sinkhorn-JA     & 106  & 94.3  & 61 & 93.4 &\multicolumn{2}{|c|}{-}\\
						    &PM-SDP & 106  & 94.3  & 61 & 93.4 &\multicolumn{2}{|c|}{-}\\
						    &Human-expert & 106  & 94.3  & 61 & 93.4 &\multicolumn{2}{|c|}{-}\\
						    
	\hline
	\multirow{2}{*}{Above Family} &Sinkhorn-JA     & 116  & \textbf{99.1}  & 61 & 100   &\multicolumn{2}{|c|}{-}\\
					        &PM-SDP & 116  & 98.2  & 61 & 100   &\multicolumn{2}{|c|}{-}\\
					        &Human-expert & 116  & 95.7  & 61 & 100   &\multicolumn{2}{|c|}{-}\\				
	\hline
\end{tabular}

\section{Conclusion}
In this paper, we suggested a new algorithm for approximately solving the JA relaxation, by generalizing the Sinkhorn algorithm for higher dimensional polytopes. This algorithm is significantly more scalable than standard solvers, and as a result the high quality solutions often obtained by the JA relaxations are made available for problems of non-trivial size. 

The main drawback of our algorithm is the fact that we only approximate the optimal solution, but as we demostrate in the results section, it nevertheless achieves state of the art performance on various tasks. We believe that other lifted relaxations can benefit from such Sinkhorn-like  solvers and leave it a possible future work direction.  
 \section*{Acknowledgments}
This research was supported in part by the European Research Council (ERC Consolidator Grant, "LiftMatch" 771136) and the Israel Science Foundation (Grant No. 1830/17). 

\bibliographystyle{unsrt}
\bibliography{nipsBib}



\end{document}